\documentclass[12pt,a4paper]{amsart}
\usepackage{amssymb,amsxtra}
\usepackage[english,french]{babel}

\newtheorem{thm}{Theorem}[section]
\newtheorem{lem}[thm]{Lemma}

\theoremstyle{definition}

\theoremstyle{remark}

\theoremstyle{plain}

\theoremstyle{remark}

\newtheorem*{example}{Example}
\numberwithin{equation}{section}

\begin{document}

\title{ Enumeration of some particular   sextuple  Persymmetric  Matrices over $\mathbb{F}_{2} $ by rank}
\author{Jorgen~Cherly}
\address{D\'epartement de Math\'ematiques, Universit\'e de
    Brest, 29238 Brest cedex~3, France}
\email{Jorgen.Cherly@univ-brest.fr}
\email{andersen69@wanadoo.fr}

\maketitle 
\begin{abstract}
Dans cet article nous comptons le nombre de certaines  sextuples  matrices persym\' etriques de rang i sur $ \mathbb {F} _ {2} . $

 \end{abstract}

\selectlanguage{english}

\begin{abstract}
In this paper we count the number of some particular  sextuple  persymmetric rank i matrices over  $ \mathbb{F}_{2}.$
 \end{abstract}
 
  \maketitle 
\newpage
\tableofcontents
\newpage

  \allowdisplaybreaks

   \section{Introduction}
  \label{sec 1}  
  In this paper we propose to compute in the most simple case  the number of  sextuple persymmetric 
  matrices with entries in   $ \mathbb{F}_{2}$ of rank i\\
  That is to compute the number  $ \Gamma_{i}^{\left[2\atop{2 \atop{ 2\atop {2\atop{2\atop2}}}}\right]\times k}$ 
  of  sextuple  persymmetric matrices in   $ \mathbb{F}_{2}$ of rank i  $(0\leqslant i\leqslant\inf(12,k) )$ of the below form.\\
    \begin{equation}
    \label{eq 1.1}
   \left (  \begin{array} {ccccccc}
\alpha  _{1}^{(1)} & \alpha  _{2}^{(1)}  &   \alpha_{3}^{(1)} &   \alpha_{4}^{(1)} &   \alpha_{5}^{(1)}  & \ldots  &  \alpha_{k}^{(1)} \\
\alpha  _{2}^{(1)} & \alpha  _{3}^{(1)}  &   \alpha_{4}^{(1)} &   \alpha_{5}^{(1)} &   \alpha_{6}^{(1)}  & \ldots  &  \alpha_{k+1}^{(1)} \\ 
\hline \\
\alpha  _{1}^{(2)} & \alpha  _{2}^{(2)}  &   \alpha_{3}^{(2)} &   \alpha_{4}^{(2)} &   \alpha_{5}^{(2)} & \ldots   &  \alpha_{k}^{(2)} \\
\alpha  _{2}^{(2)} & \alpha  _{3}^{(2)}  &   \alpha_{4}^{(2)} &   \alpha_{5}^{(2)}&   \alpha_{6}^{(2)}   & \ldots  &  \alpha_{k+1}^{(2)} \\ 
\hline\\
\alpha  _{1}^{(3)} & \alpha  _{2}^{(3)}  &   \alpha_{3}^{(3)}  &   \alpha_{4}^{(3)} &   \alpha_{5}^{(3)}  & \ldots  &  \alpha_{k}^{(3)} \\
\alpha  _{2}^{(3)} & \alpha  _{3}^{(3)}  &   \alpha_{4}^{(3)}&   \alpha_{5}^{(3)} &   \alpha_{6}^{(3)}  & \ldots  &  \alpha_{k+1}^{(3)} \\ 
\hline \\
\alpha  _{1}^{(4)} & \alpha  _{2}^{(4)}  &   \alpha_{3}^{(4)} &   \alpha_{4}^{(4)} &   \alpha_{5}^{(4)}  & \ldots  &  \alpha_{k}^{(4)} \\
\alpha  _{2}^{(4)} & \alpha  _{3}^{(4)}  &   \alpha_{4}^{(4)}&   \alpha_{5}^{(4)} &   \alpha_{6}^{(4)}  & \ldots  &  \alpha_{k+1}^{(4)} \\ 
\hline \\
\alpha  _{1}^{(5)} & \alpha  _{2}^{(5)}  &   \alpha_{3}^{(5)} &   \alpha_{4}^{(5)} &   \alpha_{5}^{(5)}  & \ldots  &  \alpha_{k}^{(5)} \\
\alpha  _{2}^{(5)} & \alpha  _{3}^{(5)}  &   \alpha_{4}^{(5)}&   \alpha_{5}^{(5)} &   \alpha_{6}^{(5)}  & \ldots  &  \alpha_{k+1}^{(5)}\\
\hline \\
\alpha  _{1}^{(6)} & \alpha  _{2}^{(6)}  &   \alpha_{3}^{(6)} &   \alpha_{4}^{(6)} &   \alpha_{5}^{(6)}  & \ldots  &  \alpha_{k}^{(6)} \\
\alpha  _{2}^{(6)} & \alpha  _{3}^{(6)}  &   \alpha_{4}^{(6)}&   \alpha_{5}^{(6)} &   \alpha_{6}^{(6)}  & \ldots  &  \alpha_{k+1}^{(6)}\\
\end{array} \right )  
\end{equation} 
We remark that this paper is just a generalization of  the results obtained in  the author's paper [13] concerning quintuple persymmetric matrices
in   $ \mathbb{F}_{2}.$

\section{Notations and Preliminaries}
\label{sec 2}
 \subsection{Some notations concerning the field of Laurent Series $ \mathbb{F}_{2}((T^{-1})) $ }
   We denote by $ \mathbb{F}_{2}\big(\big({T^{-1}}\big) \big)
 = \mathbb{K} $ the completion
 of the field $\mathbb{F}_{2}(T), $  the field of  rational fonctions over the
 finite field\; $\mathbb{F}_{2}$,\; for the  infinity  valuation \;
 $ \mathfrak{v}=\mathfrak{v}_{\infty }$ \;defined by \;
 $ \mathfrak{v}\big(\frac{A}{B}\big) = degB -degA $ \;
 for each pair (A,B) of non-zero polynomials.
 Then every element non-zero t in
  $\mathbb{F}_{2}\big(\big({\frac{1}{T}}\big) \big) $
 can be expanded in a unique way in a convergent Laurent series
                              $  t = \sum_{j= -\infty }^{-\mathfrak{v}(t)}t_{j}T^j
                                 \; where\; t_{j}\in \mathbb{F}_{2}. $\\
  We associate to the infinity valuation\; $\mathfrak{v}= \mathfrak{v}_{\infty }$
   the absolute value \; $\vert \cdot \vert_{\infty} $\; defined by \;
  \begin{equation*}
  \vert t \vert_{\infty} =  \vert t \vert = 2^{-\mathfrak{v}(t)}. \\
\end{equation*}
    We denote  E the  Character of the additive locally compact group
$  \mathbb{F}_{2}\big(\big({\frac{1}{T}}\big) \big) $ defined by \\
\begin{equation*}
 E\big( \sum_{j= -\infty }^{-\mathfrak{v}(t)}t_{j}T^j\big)= \begin{cases}
 1 & \text{if      }   t_{-1}= 0, \\
  -1 & \text{if      }   t_{-1}= 1.
    \end{cases}
\end{equation*}
  We denote $\mathbb{P}$ the valuation ideal in $ \mathbb{K},$ also denoted the unit interval of  $\mathbb{K},$ i.e.
  the open ball of radius 1 about 0 or, alternatively, the set of all Laurent series 
   $$ \sum_{i\geq 1}\alpha _{i}T^{-i}\quad (\alpha _{i}\in  \mathbb{F}_{2} ) $$ and, for every rational
    integer j,  we denote by $\mathbb{P}_{j} $
     the  ideal $\left\{t \in \mathbb{K}|\; \mathfrak{v}(t) > j \right\}. $
     The sets\; $ \mathbb{P}_{j}$\; are compact subgroups  of the additive
     locally compact group \; $ \mathbb{K}. $\\
      All $ t \in \mathbb{F}_{2}\Big(\Big(\frac{1}{T}\Big)\Big) $ may be written in a unique way as
$ t = [t] + \left\{t\right\}, $ \;  $  [t] \in \mathbb{F}_{2}[T] ,
 \; \left\{t\right\}\in \mathbb{P}  ( =\mathbb{P}_{0}). $\\
 We denote by dt the Haar measure on \; $ \mathbb{K} $\; chosen so that \\
  $$ \int_{\mathbb{P}}dt = 1. $$\\
  
  $$ Let \quad
  (t_{1},t_{2},\ldots,t_{n} )
 =  \big( \sum_{j=-\infty}^{-\nu(t_{1})}\alpha _{j}^{(1)}T^{j},  \sum_{j=-\infty}^{-\nu(t_{2})}\alpha _{j}^{(2)}T^{j} ,\ldots, \sum_{j=-\infty}^{-\nu(t_{n})}\alpha _{j}^{(n)}T^{j}\big) \in  \mathbb{K}^{n}. $$ 
 We denote $\psi  $  the  Character on  $(\mathbb{K}^n, +) $ defined by \\
 \begin{align*}
  \psi \big( \sum_{j=-\infty}^{-\nu(t_{1})}\alpha _{j}^{(1)}T^{j},  \sum_{j=-\infty}^{-\nu(t_{2})}\alpha _{j}^{(2)}T^{j} ,\ldots, \sum_{j=-\infty}^{-\nu(t_{n})}\alpha _{j}^{(n)}T^{j}\big) & = E \big( \sum_{j=-\infty}^{-\nu(t_{1})}\alpha _{j}^{(1)}T^{j}\big) \cdot E\big( \sum_{j=-\infty}^{-\nu(t_{2})}\alpha _{j}^{(2)}T^{j}\big)\cdots E\big(  \sum_{j=-\infty}^{-\nu(t_{n})}\alpha _{j}^{(n)}T^{j}\big) \\
  & = 
    \begin{cases}
 1 & \text{if      }     \alpha _{-1}^{(1)} +    \alpha _{-1}^{(2)}  + \ldots +   \alpha _{-1}^{(n)}   = 0 \\
  -1 & \text{if      }    \alpha _{-1}^{(1)} +    \alpha _{-1}^{(2)}  + \ldots +   \alpha _{-1}^{(n)}   =1                                                                                                                          
    \end{cases}
  \end{align*}

   \subsection{Some results concerning  n-times persymmetric matrices over  $ \mathbb{F}_{2}$}
 \label{subsec 2.2}
     $$ Set\quad
  (t_{1},t_{2},\ldots,t_{n} )
 =  \big( \sum_{i\geq 1}\alpha _{i}^{(1)}T^{-i}, \sum_{i \geq 1}\alpha  _{i}^{(2)}T^{-i},\sum_{i \geq 1}\alpha _{i}^{(3)}T^{-i},\ldots,\sum_{i \geq 1}\alpha _{i}^{(n)}T^{-i}   \big) \in  \mathbb{P}^{n}. $$

     Denote by $D^{\left[2 \atop{\vdots \atop 2}\right]\times k}(t_{1},t_{2},\ldots,t_{n} ) $
    
    the following $2n \times k $ \;  n-times  persymmetric  matrix  over the finite field  $\mathbb{F}_{2} $ 
    
  \begin{equation}
  \label{eq 2.1}
   \left (  \begin{array} {cccccccc}
\alpha  _{1}^{(1)} & \alpha  _{2}^{(1)}  &   \alpha_{3}^{(1)} &   \alpha_{4}^{(1)} &   \alpha_{5}^{(1)} &  \alpha_{6}^{(1)}  & \ldots  &  \alpha_{k}^{(1)} \\
\alpha  _{2}^{(1)} & \alpha  _{3}^{(1)}  &   \alpha_{4}^{(1)} &   \alpha_{5}^{(1)} &   \alpha_{6}^{(1)} &  \alpha_{7}^{(1)} & \ldots  &  \alpha_{k+1}^{(1)} \\ 
\hline \\
\alpha  _{1}^{(2)} & \alpha  _{2}^{(2)}  &   \alpha_{3}^{(2)} &   \alpha_{4}^{(2)} &   \alpha_{5}^{(2)} &  \alpha_{6}^{(2)} & \ldots   &  \alpha_{k}^{(2)} \\
\alpha  _{2}^{(2)} & \alpha  _{3}^{(2)}  &   \alpha_{4}^{(2)} &   \alpha_{5}^{(2)}&   \alpha_{6}^{(2)} &  \alpha_{7}^{(2)}  & \ldots  &  \alpha_{k+1}^{(2)} \\ 
\hline\\
\alpha  _{1}^{(3)} & \alpha  _{2}^{(3)}  &   \alpha_{3}^{(3)}  &   \alpha_{4}^{(3)} &   \alpha_{5}^{(3)} &  \alpha_{6}^{(3)} & \ldots  &  \alpha_{k}^{(3)} \\
\alpha  _{2}^{(3)} & \alpha  _{3}^{(3)}  &   \alpha_{4}^{(3)}&   \alpha_{5}^{(3)} &   \alpha_{6}^{(3)}  &  \alpha_{7}^{(3)} & \ldots  &  \alpha_{k+1}^{(3)} \\ 
\hline \\
\vdots & \vdots & \vdots  & \vdots  & \vdots & \vdots  & \vdots & \vdots \\
\hline \\
\alpha  _{1}^{(n)} & \alpha  _{2}^{(n)}  &   \alpha_{3}^{(n)} &   \alpha_{4}^{(n)} &   \alpha_{5}^{(n)}  &  \alpha_{6}^{(n)} & \ldots  &  \alpha_{k}^{(n)} \\
\alpha  _{2}^{(n)} & \alpha  _{3}^{(n)}  &   \alpha_{4}^{(n)}&   \alpha_{5}^{(n)} &   \alpha_{6}^{(n)}  &  \alpha_{7}^{(n)} & \ldots  &  \alpha_{k+1}^{(n)} \\ 
\end{array} \right )  
\end{equation} 
We denote by  $ \Gamma_{i}^{\left[2\atop{\vdots \atop 2}\right]\times k}$  the number of rank i n-times persymmetric matrices over $\mathbb{F}_{2}$ of the above form :  \\

  Let $ \displaystyle  f (t_{1},t_{2},\ldots,t_{n} ) $  be the exponential sum  in $ \mathbb{P}^{n} $ defined by\\
    $(t_{1},t_{2},\ldots,t_{n} ) \displaystyle\in \mathbb{P}^{n}\longrightarrow \\
    \sum_{deg Y\leq k-1}\sum_{deg U_{1}\leq  1}E(t_{1} YU_{1})
  \sum_{deg U_{2} \leq 1}E(t _{2} YU_{2}) \ldots \sum_{deg U_{n} \leq 1} E(t _{n} YU_{n}). $\vspace{0.5 cm}\\
    Then
  $$     f_{k} (t_{1},t_{2},\ldots,t_{n} ) =
  2^{2n+k- rank\big[ D^{\left[2\atop{\vdots \atop 2}\right]\times k}(t_{1},t_{2},\ldots,t_{n} )\big] } $$

    Hence  the number denoted by $ R_{q,n}^{(k)} $ of solutions \\
  
 $(Y_1,U_{1}^{(1)},U_{2}^{(1)}, \ldots,U_{n}^{(1)}, Y_2,U_{1}^{(2)},U_{2}^{(2)}, 
\ldots,U_{n}^{(2)},\ldots  Y_q,U_{1}^{(q)},U_{2}^{(q)}, \ldots,U_{n}^{(q)}   ) \in (\mathbb{F}_{2}[T])^{(n+1)q}$ \vspace{0.5 cm}\\
 of the polynomial equations  \vspace{0.5 cm}
  \[\left\{\begin{array}{c}
 Y_{1}U_{1}^{(1)} + Y_{2}U_{1}^{(2)} + \ldots  + Y_{q}U_{1}^{(q)} = 0  \\
    Y_{1}U_{2}^{(1)} + Y_{2}U_{2}^{(2)} + \ldots  + Y_{q}U_{2}^{(q)} = 0\\
    \vdots \\
   Y_{1}U_{n}^{(1)} + Y_{2}U_{n}^{(2)} + \ldots  + Y_{q}U_{n}^{(q)} = 0 
 \end{array}\right.\]
 
    $ \Leftrightarrow
    \begin{pmatrix}
   U_{1}^{(1)} & U_{1}^{(2)} & \ldots  & U_{1}^{(q)} \\ 
      U_{2}^{(1)} & U_{2}^{(2)}  & \ldots  & U_{2}^{(q)}  \\
\vdots &   \vdots & \vdots &   \vdots   \\
U_{n}^{(1)} & U_{n}^{(2)}   & \ldots  & U_{n}^{(q)} \\
 \end{pmatrix}  \begin{pmatrix}
   Y_{1} \\
   Y_{2}\\
   \vdots \\
   Y_{q} \\
  \end{pmatrix} =   \begin{pmatrix}
  0 \\
  0 \\
  \vdots \\
  0 
  \end{pmatrix} $\\
  satisfying the degree conditions \\
                   $$  degY_i \leq k-1 ,
                   \quad degU_{j}^{(i)} \leq 1, \quad  for \quad 1\leq j\leq n  \quad 1\leq i \leq q $$ \\
  is equal to the following integral over the unit interval in $ \mathbb{K}^{n} $
    $$ \int_{\mathbb{P}^{n}} f_{k}^{q}(t_{1},t_{2},\ldots,t_{n}) dt_{1}dt _{2}\ldots dt _{n}. $$
  Observing that $ f (t_{1},t_{2},\ldots,t_{n} ) $ is constant on cosets of $ \prod_{j=1}^{n}\mathbb{P}_{k+1} $ in $ \mathbb{P}^{n} $\;
  the above integral is equal to 
  
  \begin{equation}
  \label{eq 2.2}
 2^{q(2n+k) - (k+1)n}\sum_{i = 0}^{\inf(2n,k)}
  \Gamma_{i}^{\left[2\atop{\vdots \atop 2}\right]\times k} 2^{-iq} =  R_{q,n}^{(k)} 
 \end{equation}
 
 \begin{eqnarray}
 \label{eq 2.3}
\text{ Recall that $ R_{q,n}^{(k)}$ is equal to the number of solutions of the polynomial system} \nonumber \\
    \begin{pmatrix}
   U_{1}^{(1)} & U_{1}^{(2)} & \ldots  & U_{1}^{(q)} \\ 
      U_{2}^{(1)} & U_{2}^{(2)}  & \ldots  & U_{2}^{(q)}  \\
\vdots &   \vdots & \vdots &   \vdots   \\
U_{n}^{(1)} & U_{n}^{(2)}   & \ldots  & U_{n}^{(q)} \\
 \end{pmatrix}  \begin{pmatrix}
   Y_{1} \\
   Y_{2}\\
   \vdots \\
   Y_{q} \\
  \end{pmatrix} =   \begin{pmatrix}
  0 \\
  0 \\
  \vdots \\
  0 
  \end{pmatrix} \\
 \text{ satisfying the degree conditions}\nonumber \\
                     degY_i \leq k-1 ,
                   \quad degU_{j}^{(i)} \leq 1, \quad  for \quad 1\leq j\leq n  \quad 1\leq i \leq q  \nonumber
 \end{eqnarray}
  From \eqref{eq 2.2} we obtain for q = 1\\
   \begin{align}
  \label{eq 2.4}
 2^{k-(k-1)n}\sum_{i = 0}^{\inf(2n,k)}
 \Gamma_{i}^{\left[2\atop{\vdots \atop 2}\right]\times k} 2^{-i} =  R_{1,n}^{(k)} = 2^{2n}+2^k-1
  \end{align}
We have obviously \\
  \begin{align}
  \label{eq 2.5}
 \sum_{i = 0}^{k}
 \Gamma_{i}^{\left[2\atop{\vdots \atop 2}\right]\times k}  = 2^{(k+1)n}  
 \end{align}
From  the fact that the number of rank one persymmetric  matrices over $\mathbb{F}_{2}$ is equal to three  we obtain using
 combinatorial methods  : \\
   \begin{equation}
  \label{eq 2.6}
 \Gamma_{1}^{\left[2\atop{\vdots \atop 2}\right]\times k}  = (2^{n}-1)\cdot 3
  \end{equation}
  For more details see Cherly  [11,12].
  \subsection{The case n=6}
   \begin{eqnarray*}
       Set\quad
  (t_{1},t_{2},t_{3},t_{4},t_{5},t_{6} )\hspace{12 cm}\\
 =  \big( \sum_{i\geq 1}\alpha _{i}^{(1)}T^{-i}, \sum_{i \geq 1}\alpha  _{i}^{(2)}T^{-i},\sum_{i \geq 1}\alpha _{i}^{(3)}T^{-i},\sum_{i \geq 1}\alpha _{i}^{(4)}T^{-i},\sum_{i \geq 1}\alpha _{i}^{(5)}T^{-i}, \sum_{i \geq 1}\alpha _{i}^{(6)}T^{-i}  \big) \in  \mathbb{P}^{6}. 
  \end{eqnarray*}
  Denote by $D^{\left[2 \atop{2\atop{2 \atop {2\atop {2\atop 2}}}}\right]\times k}(t_{1},t_{2},t_{3},t_{4},t_{5},t_{6}) $
    
    the following $12 \times k $ \; sextuple  persymmetric  matrix  over the finite field  $\mathbb{F}_{2} $ 
   
  \begin{displaymath}
   \left (  \begin{array} {cccccccc}
\alpha  _{1}^{(1)} & \alpha  _{2}^{(1)}  &   \alpha_{3}^{(1)} &   \alpha_{4}^{(1)} &   \alpha_{5}^{(1)}   & \ldots  &  \alpha_{k}^{(1)} \\
\alpha  _{2}^{(1)} & \alpha  _{3}^{(1)}  &   \alpha_{4}^{(1)} &   \alpha_{5}^{(1)} &   \alpha_{6}^{(1)} & \ldots  &  \alpha_{k+1}^{(1)} \\ 
\hline \\
\alpha  _{1}^{(2)} & \alpha  _{2}^{(2)}  &   \alpha_{3}^{(2)} &   \alpha_{4}^{(2)} &   \alpha_{5}^{(2)} & \ldots   &  \alpha_{k}^{(2)} \\
\alpha  _{2}^{(2)} & \alpha  _{3}^{(2)}  &   \alpha_{4}^{(2)} &   \alpha_{5}^{(2)}&   \alpha_{6}^{(2)}   & \ldots  &  \alpha_{k+1}^{(2)} \\ 
\hline\\
\alpha  _{1}^{(3)} & \alpha  _{2}^{(3)}  &   \alpha_{3}^{(3)}  &   \alpha_{4}^{(3)} &   \alpha_{5}^{(3)} & \ldots  &  \alpha_{k}^{(3)} \\
\alpha  _{2}^{(3)} & \alpha  _{3}^{(3)}  &   \alpha_{4}^{(3)}&   \alpha_{5}^{(3)} &   \alpha_{6}^{(3)} & \ldots  &  \alpha_{k+1}^{(3)} \\ 
\hline \\
\alpha  _{1}^{(4)} & \alpha  _{2}^{(4)}  &   \alpha_{3}^{(4)} &   \alpha_{4}^{(4)} &   \alpha_{5}^{(4)}  & \ldots  &  \alpha_{k}^{(4)} \\
\alpha  _{2}^{(4)} & \alpha  _{3}^{(4)}  &   \alpha_{4}^{(4)}&   \alpha_{5}^{(4)} &   \alpha_{6}^{(4)}  & \ldots  &  \alpha_{k+1}^{(4)} \\ 
\hline \\
\alpha  _{1}^{(5)} & \alpha  _{2}^{(5)}  &   \alpha_{3}^{(5)} &   \alpha_{4}^{(5)} &   \alpha_{5}^{(5)} & \ldots  &  \alpha_{k}^{(5)} \\
\alpha  _{2}^{(5)} & \alpha  _{3}^{(5)}  &   \alpha_{4}^{(5)}&   \alpha_{5}^{(5)} &   \alpha_{6}^{(5)}  & \ldots  &  \alpha_{k+1}^{(5)} \\
\hline \\
\alpha  _{1}^{(6)} & \alpha  _{2}^{(6)}  &   \alpha_{3}^{(6)} &   \alpha_{4}^{(6)} &   \alpha_{5}^{(6)} & \ldots  &  \alpha_{k}^{(6)} \\
\alpha  _{2}^{(6)} & \alpha  _{3}^{(6)}  &   \alpha_{4}^{(6)}&   \alpha_{5}^{(6)} &   \alpha_{6}^{(6)}  & \ldots  &  \alpha_{k+1}^{(6)} \\
\end{array} \right )  
\end{displaymath} 
We denote by  $ \Gamma_{i}^{\left[2\atop{2 \atop{ 2\atop {2\atop {2\atop 2}}}}\right]\times k}$  the number of rank i sextuple persymmetric matrices over $\mathbb{F}_{2}$ of the above form :  \\

  Let $ \displaystyle  f (t_{1},t_{2},t_{3},t_{4},t_{5},t_{6} ) $  be the exponential sum  in $ \mathbb{P}^{6} $ defined by\\
    $(t_{1},t_{2},t_{3},t_{4},t_{5},t_{6}) \displaystyle\in \mathbb{P}^{6}\longrightarrow 
    \sum_{deg Y\leq k-1}\sum_{deg U_{1}\leq  1}E(t_{1} YU_{1})
  \sum_{deg U_{2} \leq 1}E(t _{2} YU_{2}) \sum_{deg U_{3} \leq 1}E(t _{3} YU_{3}) \\ 
  \sum_{deg U_{4} \leq 1} E(t _{4} YU_{4}) \sum_{deg U_{5} \leq 1} E(t _{5} YU_{5} \sum_{deg U_{6} \leq 1} E(t _{6} YU_{6}).$\\
 \vspace{0.5 cm}\\
    Then
  $$     f_{k} (t_{1},t_{2},t_{3},t_{4},t_{5},t_{6} ) =
  2^{12+k- rank\big[ D^{\left[2\atop{2\atop{2 \atop {2\atop {2\atop}}}}\right]\times k}(t_{1},t_{2},t_{3},t_{4},t_{5},t_{6} )\big] } $$

 Hence  the number denoted by $ R_{q,6}^{(k)} $ of solutions \\
  
 $(Y_1,U_{1}^{(1)},U_{2}^{(1)},U_{3}^{(1)} ,U_{4}^{(1)}, U_{5}^{(1)},U_{6}^{(1)},Y_2,U_{1}^{(2)},U_{2}^{(2)}, 
U_{3}^{(2)},U_{4}^{(2)},U_{5}^{(2)},U_{6}^{(2)}\\
\ldots  Y_q,U_{1}^{(q)},U_{2}^{(q)}, U_{3}^{(q)},U_{4}^{(q)},U_{5}^{(q)},U_{6}^{(q)}   ) \in (\mathbb{F}_{2}[T])^{7q}$ \vspace{0.5 cm}\\
 of the polynomial equations  \vspace{0.5 cm}
  \[\left\{\begin{array}{c}
 Y_{1}U_{1}^{(1)} + Y_{2}U_{1}^{(2)} + \ldots  + Y_{q}U_{1}^{(q)} = 0  \\
    Y_{1}U_{2}^{(1)} + Y_{2}U_{2}^{(2)} + \ldots  + Y_{q}U_{2}^{(q)} = 0\\
    Y_{1}U_{3}^{(1)} + Y_{3}U_{3}^{(2)} + \ldots  + Y_{q}U_{3}^{(q)} = 0\\ 
   Y_{1}U_{4}^{(1)} + Y_{2}U_{4}^{(2)} + \ldots  + Y_{q}U_{4}^{(q)} = 0 \\
  Y_{1}U_{5}^{(1)} + Y_{2}U_{5}^{(2)} + \ldots  + Y_{q}U_{5}^{(q)} = 0 \\
   Y_{1}U_{6}^{(1)} + Y_{2}U_{6}^{(2)} + \ldots  + Y_{q}U_{6}^{(q)} = 0 \\
  \end{array}\right.\]
 
    $ \Leftrightarrow
    \begin{pmatrix}
   U_{1}^{(1)} & U_{1}^{(2)} & \ldots  & U_{1}^{(q)} \\ 
      U_{2}^{(1)} & U_{2}^{(2)}  & \ldots  & U_{2}^{(q)}  \\
 U_{3}^{(1)} & U_{3}^{(2)}  & \ldots  & U_{3}^{(q)}  \\
U_{4}^{(1)} & U_{4}^{(2)}   & \ldots  & U_{4}^{(q)} \\
U_{5}^{(1)} & U_{5}^{(2)}   & \ldots  & U_{5}^{(q)}\\
U_{6}^{(1)} & U_{6}^{(2)}   & \ldots  & U_{6}^{(q)}
 \end{pmatrix}  \begin{pmatrix}
   Y_{1} \\
   Y_{2}\\
   \vdots \\
   Y_{q} \\
  \end{pmatrix} =   \begin{pmatrix}
  0 \\
  0 \\
  0 \\
  0 \\
  0\\
  0
  \end{pmatrix} $\\
  satisfying the degree conditions \\
                   $$  degY_i \leq k-1 ,
                   \quad degU_{j}^{(i)} \leq 1, \quad  for \quad 1\leq j\leq 6  \quad 1\leq i \leq q $$ \\
  is equal to the following integral over the unit interval in $ \mathbb{K}^{6} $
    $$ \int_{\mathbb{P}^{6}} f_{k}^{q}(t_{1},t_{2},t_{3},t_{4},t_{5},t_{6}) dt_{1}dt _{2}dt_{3} dt _{4}dt_{5}dt_{6}$$
  Observing that $ f (t_{1},t_{2},t_{3},t_{4},t_{5},t_{6} ) $ is constant on cosets of $ \prod_{j=1}^{6}\mathbb{P}_{k+1} $ in $ \mathbb{P}^{6} $\;
  the above integral is equal to 
   \begin{equation}
  \label{eq 2.7}
 2^{q(12+k) - 6(k+1)}\sum_{i = 0}^{\inf{(12,k)}}
  \Gamma_{i}^{\left[2\atop{2\atop{2 \atop {2\atop {2\atop2}}}}\right]\times k} 2^{-iq} =  R_{q,6}^{(k)} \quad \text{where} \; k\geqslant 1
 \end{equation}
   \subsection{Some preliminary results}
 \begin{lem}
\label{lem 2.1}
   \begin{equation}
  \label{eq 2.8}
  \begin{cases} 
 \displaystyle   \Gamma_{0}^{\left[2\atop{\vdots \atop 2}\right]\times k}    = 1 \quad \text{if} \quad  k\geqslant 1 \\
  \displaystyle    \Gamma_{1}^{\left[2\atop{\vdots \atop 2}\right]\times k}    = 3\cdot(2^n-1) \quad \text{if} \quad  k\geqslant 2 \\
 \displaystyle    \Gamma_{2}^{\left[2\atop{\vdots \atop 2}\right]\times k} = (2^{n+1}-2)\cdot2^{k}+7\cdot2^{2n}-25\cdot2^{n}+18 \quad \text{for} \quad k\geqslant 3\\
\displaystyle   \Gamma_{3}^{\left[2\atop{\vdots \atop 2}\right]\times k} = [7\cdot2^{2n}-21\cdot2^{n}+14]\cdot2^{k}+15\cdot2^{3n}-133\cdot2^{2n}+294\cdot2^{n}-176  \quad \text{for} \quad k\geqslant 4\\
\displaystyle    \Gamma_{4}^{\left[2\atop{\vdots \atop 2}\right]\times k} = \frac{1}{3}\cdot(2^{2n+1}-6\cdot2^{n}+4)\cdot2^{2k}\\+
\frac{1}{6}\cdot(105\cdot2^{3n}-783\cdot2^{2n}+1614\cdot2^{n}-936)\cdot2^{k}\\
+\frac{1}{6}\cdot(186\cdot2^{4n}-3630\cdot2^{3n}+19028\cdot2^{2n}-34464\cdot2^{n}+18880)
\quad \text{for} \quad k\geqslant 5\\
\displaystyle  \Gamma_{5}^{\left[2\atop{\vdots \atop 2}\right]\times k}  =
  \frac{1}{2}\cdot(5\cdot2^{3n}-35\cdot2^{2n}+70\cdot2^{n}-40)\cdot2^{2k}\\+
    \frac{1}{4}\cdot(155\cdot2^{4n}-2565\cdot2^{3n}+12530\cdot2^{2n}-21960\cdot2^{n}+11840)\cdot2^{k}\\
  + 63\cdot2^{5n}-2573\cdot2^{4n}+29150\cdot2^{3n}-123760\cdot2^{2n}+203872\cdot2^{n}-106752
 \quad \text{for} \quad k\geqslant 6\\
 \Gamma_{6}^{\left[2\atop{\vdots \atop 2}\right]\times k} = 
    \frac{1}{21}\cdot(2^{3n}-7\cdot2^{2n}+14\cdot2^{n}-8)\cdot2^{3k}\\
    + \frac{1}{168}\cdot(1085\cdot2^{4n}-16723\cdot2^{3n}+79086\cdot2^{2n}-136472\cdot2^{n}+73024)\cdot2^{2k}\\
     + \frac{1}{168}\cdot(13671\cdot2^{5n}-475881\cdot2^{4n}+5026378\cdot2^{3n}-20647816\cdot2^{2n}+33473216\cdot2^{n}-17389568)\cdot2^{k}\\
      + \frac{1}{168}\cdot(21336\cdot2^{6n}-1781640\cdot2^{5n}+41896624\cdot2^{4n}\\
      -382091648\cdot2^{3n}+1470524160\cdot2^{2n}-2311493632\cdot2^{n}+1182924800)   \quad \text{for} \quad k\geqslant 7\\   \\
   \displaystyle   \Gamma_{7}^{\left[2\atop{\vdots \atop 2}\right]\times k}  =  \frac{31}{168}\cdot [2^{4n}-15\cdot 2^{3n}+70 \cdot 2^{2n}-120\cdot 2^{n}+64] \cdot2^{3k}\\
+  \frac{1}{96}\cdot [ 1395\cdot2^{5n}-45229\cdot2^{4n}+462210\cdot 2^{3n}-1868680 \cdot 2^{2n}+3005760\cdot 2^{n}-1555456] \cdot2^{2k}\\
+  \frac{1}{48}\cdot [ 8001\cdot2^{6n}-571023\cdot2^{5n}+12524806\cdot2^{4n}-110524920\cdot 2^{3n}+418606144\cdot 2^{2n}\\
-652818432\cdot 2^{n}+332775424] \cdot2^{k}\\
+\frac{1}{21}\cdot [5355\cdot2^{7n}- 904113\cdot2^{6n}+43302294\cdot2^{5n}-817168432\cdot2^{4n}+6743660640\cdot 2^{3n}\\
-96649567\cdot2^{8}\cdot 2^{2n}+4637778\cdot2^{13}\cdot 2^{n}-293263\cdot2^{16}] \quad \text{for} \quad k\geqslant 8\\ 
\end{cases}
    \end{equation}
\end{lem}
\begin{proof}
  We recall the result obtained in Lemma 3.3 [15] concerning the number of rank i n-times persymmetric matrices over $ \mathbb{F}_{2} $
  of the form \eqref{eq 2.1} for $ 0 \leqslant i \leqslant 7 $.
\end{proof}

\begin{lem}
\label{lem 2.2}
  \begin{equation}
 \label{eq 2.9}
  \begin{cases}  
\displaystyle  \sum_{i = 0}^{\inf (2n,k)} \Gamma_{i}^{\left[2\atop{\vdots \atop 2}\right]\times k}  = 2^{(k+1)n}, \\ 
  \displaystyle  \sum_{i = 0}^{\inf (2n,k)} \Gamma_{i}^{\left[2\atop{\vdots \atop 2}\right]\times k} 2^{-i}  = 2^{n+k(n-1)}+2^{(k-1)n}-2^{(k-1)n-k},\\
  \displaystyle \sum_{i = 0}^{\inf (2n,k)} \Gamma_{i}^{\left[2\atop{\vdots \atop 2}\right]\times k} 2^{-2i}  =
   2^{n+k(n-2)}+2^{-n+k(n-2)}\cdot[3\cdot2^k-3] +2^{-2n+k(n-2)}\cdot[6\cdot2^{k-1}-6] \\
   +2^{-3n+kn}-6\cdot2^{n(k-3)-k}+8\cdot2^{-3n+k(n-2)}.
\end{cases}
    \end{equation}
 \end{lem}
 \begin{proof}
Recall the relations (3.15) in [15].
\end{proof}

 \begin{lem}
\label{lem 2.3}
 The number of rank 2n n-times persymmetric matrices of the form  \eqref{eq 2.1} is equal to :\\
  \begin{equation}
 \label{eq 2.10}
 \displaystyle   \Gamma_{2n}^{\left[2\atop{\vdots \atop 2}\right]\times k}  = 2^{n}\prod_{j=1}^{n}(2^{k}-2^{2n -j}) 
 \end{equation}
 \end{lem}
  \begin{proof}
Use the formula (2.1)  in [10] with  $s_{1}=s_{2}=\ldots =s_{m} =2$ ,\; $\delta = \sum_{j=1}^{m}s_{j}=2n$ and m=n.
\end{proof}

 \begin{lem}
\label{lem 2.4} 
  We postulate that the number of sextuple persymmetric matrices of the form \eqref{eq 1.1} of rank i can be expressed in the following manner :\\
   \begin{equation}
    \label{eq 2.11}
 \Gamma_{i}^{\left[2\atop {2\atop {2\atop{2\atop {2\atop2}}}}\right]\times k} \\ = \begin{cases}
1 & \text{if  } i = 0,        \\
 a_{1} & \text{if   } i=1,\\  
 a_{2}\cdot2^{k}+ b_{2}  & \text{if   }  i = 2,  \\
 a_{3}\cdot 2^{k}+ b_{3}  & \text{if   }  i = 3, \\
  a_{4}\cdot 2^{2k} +b_{4}\cdot 2^{k}+c_{4}  & \text{if   }  i=4, \\  
  a_{5} \cdot 2^{2k}+b_{5} \cdot2^{k} +c_{5} & \text{if   }  i=5, \\
  a_{6}\cdot 2^{3k}+b_{6}\cdot2^{2k} +c_{6}\cdot 2^{k}  + d_{6}   & \text{if   }  i=6. \\
a_{7}\cdot 2^{3k}+b_{7}\cdot2^{2k}+c_{7}\cdot2^{k} +d_{7} & \text{if   }  i=7. \\
 a_{8}\cdot 2^{4k}+b_{8}\cdot 2^{3k} +c_{8}\cdot2^{2k}+d_{8}\cdot 2^{k} +e_{8} & \text{if   }  i=8.\\
     a_{9}\cdot 2^{4k}+b_{9}\cdot 2^{3k} +c_{9}\cdot2^{2k}+d_{9}\cdot 2^{k} +e_{9} & \text{if   }  i=9.\\   
a_{10}\cdot2^{5k} + b_{10}\cdot 2^{4k}+c_{10}\cdot 2^{3k} +d_{10}\cdot2^{2k}+e_{10}\cdot 2^{k} +f_{10} & \text{if   }  i=10.\\
  a_{11}\cdot2^{5k}+b_{11}\cdot2^{4k}+c_{11}\cdot2^{3k}+d_{11}\cdot2^{2k}+e_{11}\cdot2^{k}+f_{11} &  \text{if } i=11. \\
a_{12}\cdot2^{6k}+b_{12}\cdot2^{5k}+c_{12}\cdot2^{4k}+d_{12}\cdot2^{3k}+e_{12}\cdot2^{2k}+f_{12}\cdot2^{k}+g_{12} & \text{if } i=12. \\
  \end{cases}    
  \end{equation}
\end{lem}
  \begin{proof}
  To justify our assumption we recall that in the case n=2  (see [5]) we have :
   \begin{equation*}
 \Gamma_{i}^{\left[2\atop 2 \right]\times k} = 
\begin{cases}
1 & \text{if  } i = 0,\; k \geq 1,         \\
 9 & \text{if   } i=1,\; k > 1, \\
 6\cdot 2^{k} + 30 & \text{if   }  i = 2 ,\; k > 2, \\
  42\cdot 2^{k} -168 & \text{if   }  i=3 ,\; k > 3,\\ 
4\cdot 2^{2k} -48\cdot2^{k} + 128  & \text{if   }  i=4 ,\; k \geq 4
  \end{cases}    
   \end{equation*}
   and in the case n=3  (see [6] ) :
   \[ \Gamma_{i}^{\left[2\atop{ 2\atop 2} \right]\times k}=
 \begin{cases}
1  &\text{if  }  i = 0 \\
21   &\text{if  }  i = 1 \\
14\cdot2^{k} + 266  & \text{if   } i = 2 \\
294\cdot2^{k} + 1344   & \text{if   } i = 3 \\
28\cdot2^{2k} + 2604\cdot2^{k}  - 22624 & \text{if   } i = 4 \\
420\cdot2^{2k} - 10080\cdot2^{k} + 53760   & \text{if   } i = 5 \\
8\cdot2^{3k} - 448\cdot2^{2k} + 7168\cdot2^{k} - 32768  & \text{if  } i = 6,\;k\geq 6
\end{cases}
\]
 Recall similar expressions concerning quadruple, quintuple persymmetric matrices see [12,13].
  \end{proof}
  
   \begin{lem}
\label{lem 2.5} 
  We postulate : \\
   \begin{equation}
 \label{eq 2.12}
  \begin{cases} 
    \displaystyle   \Gamma_{7}^{\left[2\atop {2\atop {2\atop{2\atop{2\atop2} }}}\right]\times k}  = 0 \quad \text{for} \quad k = 6 \\
   \displaystyle   \Gamma_{8}^{\left[2\atop {2\atop {2\atop{2\atop{2\atop2} }}}\right]\times k}  = 0\quad \text{for} \quad k\in \{6,7\} \\
   \displaystyle   \Gamma_{9}^{\left[2\atop {2\atop {2\atop{2\atop{2\atop2} }}}\right]\times k} =0\quad \text{for} \quad k\in \{6,7,8\} \\
 \displaystyle   \Gamma_{10}^{\left[2\atop {2\atop {2\atop{2\atop{2\atop2} }}}\right]\times k} = 0 \quad \text{for} \quad k\in \{6,7,8,9\} \\
   \displaystyle   \Gamma_{11}^{\left[2\atop {2\atop {2\atop{2\atop{2\atop2} }}}\right]\times k}= 0 \quad \text{for} \quad k\in \{6,7,8,9,10\} \\
     \displaystyle   \Gamma_{12}^{\left[2\atop {2\atop {2\atop{2\atop{2\atop2} }}}\right]\times k}= 0 \quad \text{for} \quad k\in \{6,7,8,9,10,11\} \\  
\end{cases}
    \end{equation}
That is :\\
       \begin{equation*}
  \begin{cases} 
    \displaystyle   \Gamma_{7}^{\left[2\atop {2\atop {2\atop{2\atop{2\atop2} }}}\right]\times k}  = (2^k-2^6)\cdot(\alpha\cdot 2^{2k}+ \ldots )\\
     \displaystyle   \Gamma_{8}^{\left[2\atop {2\atop {2\atop{2\atop{2\atop2} }}}\right]\times k}  =(2^k-2^6)(2^k-2^7)\cdot(\alpha\cdot 2^{2k}+ \ldots )\\
      \displaystyle   \Gamma_{9}^{\left[2\atop {2\atop {2\atop{2\atop{2\atop2} }}}\right]\times k} =(2^k-2^6)(2^k-2^7)(2^k-2^8)\cdot(\alpha\cdot 2^{k}+ \ldots )\\
  \displaystyle   \Gamma_{10}^{\left[2\atop {2\atop {2\atop{2\atop{2\atop2} }}}\right]\times k} = (2^k-2^6)(2^k-2^7)(2^k-2^8)(2^k-2^9)\cdot(\alpha\cdot 2^{k}+ \ldots )\\
 \displaystyle   \Gamma_{11}^{\left[2\atop {2\atop {2\atop{2\atop{2\atop2} }}}\right]\times k}= \alpha(2^k-2^6)(2^k-2^7)(2^k-2^8)(2^k-2^9)(2^k-2^{10})
\\ 
  \displaystyle   \Gamma_{12}^{\left[2\atop {2\atop {2\atop{2\atop{2\atop2} }}}\right]\times k}= \alpha(2^k-2^6)(2^k-2^7)(2^k-2^8)(2^k-2^9)(2^k-2^{10})(2^k-2^{11})
   \end{cases}
    \end{equation*}
 \end{lem}
 \begin{proof}
 To justify our assumption we have in the case n=2 :\\
    \begin{equation*}
 \Gamma_{i}^{\left[2\atop 2 \right]\times k} = 
\begin{cases}
1 & \text{if  } i = 0,\; k \geq 1,         \\
 9 & \text{if   } i=1,\; k > 1, \\
 6\cdot 2^{k} + 30 & \text{if   }  i = 2 ,\; k > 2, \\
  42\cdot 2^{k} -168=42\cdot(2^k-2^2) & \text{if   }  i=3 ,\; k > 3,\\ 
4\cdot 2^{2k} -48\cdot2^{k} + 128 =4\cdot(2^k-2^2)(2^k -2^3)  & \text{if   }  i=4 ,\; k \geq 4
  \end{cases}  
    \end{equation*}
 and in the case n=3 :\\
  \[ \Gamma_{i}^{\left[2\atop{ 2\atop 2} \right]\times k}=
 \begin{cases}
1  &\text{if  }  i = 0 \\
21   &\text{if  }  i = 1 \\
14\cdot2^{k} + 266  & \text{if   } i = 2 \\
294\cdot2^{k} + 1344   & \text{if   } i = 3 \\
28\cdot2^{2k} + 2604\cdot2^{k}  - 22624 = (2^k-2^3)(28\cdot2^k+2828) & \text{if   } i = 4 \\
420\cdot2^{2k} - 10080\cdot2^{k} + 53760 =420\cdot(2^k-2^3)(2^k-2^4)  & \text{if   } i = 5 \\
8\cdot2^{3k} - 448\cdot2^{2k} + 7168\cdot2^{k} - 32768 =8\cdot(2^k-2^3)(2^k-2^4)(2^k-2^5)  & \text{if  } i = 6,\;k\geq 6
\end{cases}
\]
See also the similar problem concerning quintuple persymmetric matrices [13].
  \end{proof}

   \begin{lem}
\label{lem 2.6}   
      \begin{equation}
 \label{eq 2.13}
  \begin{cases}  
\displaystyle  \sum_{i = 0}^{12}   \Gamma_{i}^{\left[2\atop {2\atop {2\atop{2\atop{2\atop2}}}}\right]\times k}   = 2^{6k+6}, \\ 
\displaystyle  \sum_{i = 0}^{12} \Gamma_{i}^{\left[2\atop {2\atop {2\atop{2\atop{2\atop2}}}}\right]\times k}   2^{12-i}  =2^{6k+6}+262080\cdot2^{5k},\\
  \displaystyle \sum_{i = 0}^{12} \Gamma_{i}^{\left[2\atop {2\atop {2\atop{2\atop{2\atop2}}}}\right]\times k} 2^{24-2i}   =2^{6k+6}+798336\cdot2^{5k}
  +1072931328\cdot2^{4k}
  \end{cases}
    \end{equation}
\end{lem}
\begin{proof}
Apply \eqref{eq 2.9} with n=6.
  \end{proof}

  \subsection{Computation of the number of sextuple persymmetric matrices of the form (1.1) of rank I}
\begin{thm}
\label{thm 2.1}
We have whenever $k\geqslant 6: $\\
 \begin{equation}
 \label{eq 2.14}
  \begin{cases} 
 \displaystyle   \Gamma_{0}^{\left[2\atop {2\atop {2\atop{2\atop{2\atop2}}}}\right]\times k}   = 1 \quad \text{if} \quad  k\geqslant 1 \\
 \displaystyle    \Gamma_{1}^{\left[2\atop {2\atop {2\atop{2\atop{2\atop2}}}}\right]\times k}   = 189 \quad \text{if} \quad  k\geqslant 2 \\
\displaystyle   \Gamma_{2}^{\left[2\atop {2\atop {2\atop{2\atop{2\atop2}}}}\right]\times k}  = 126\cdot2^{k}+27090 \quad \text{for} \quad k\geqslant 3\\
\displaystyle    \Gamma_{3}^{\left[2\atop {2\atop {2\atop{2\atop{2\atop2}}}}\right]\times k}  = 27342\cdot2^{k}+3406032 \quad \text{for} \quad k\geqslant 4\\
\displaystyle    \Gamma_{4}^{\left[2\atop {2\atop {2\atop{2\atop{2\atop2}}}}\right]\times k}  = 2604\cdot2^{2k}+4070052\cdot2^{k}+374121888\quad \text{for} \quad k\geqslant 5\\
 \displaystyle    \Gamma_{5}^{\left[2\atop {2\atop {2\atop{2\atop{2\atop2}}}}\right]\times k}  =  585900\cdot2^{2k}+494499600\cdot2^{k}+123537015\cdot2^{8}   \quad \text{for} \quad k\geqslant 6\\
   \displaystyle   \Gamma_{6}^{\left[2\atop {2\atop {2\atop{2\atop{2\atop2}}}}\right]\times k}  = 11160\cdot2^{3k}
    +84135240\cdot2^{2k}+2^8\cdot184392495\cdot2^{k}+29391255\cdot2^{15} \quad \text{for} \quad k\geqslant 7 \\
   \displaystyle   \Gamma_{7}^{\left[2\atop {2\atop {2\atop{2\atop{2\atop2}}}}\right]\times k} 
    = 2421720\cdot2^{3k} +277589655\cdot2^5\cdot2^{2k} +2431729125\cdot2^{10}\cdot2^{k}\\
     -2996595315\cdot2^{16} \quad \text{for} \quad k\geqslant 8 \\
   \end{cases}
    \end{equation}  
     \begin{equation}
     \label{eq 2.15}
  \begin{cases} 
 \displaystyle   \Gamma_{8}^{\left[2\atop {2\atop {2\atop{2\atop{2\atop2}}}}\right]\times k}  = 10416\cdot2^{4k}+216944\cdot1395\cdot2^{3k}
  +2155757205\cdot2^{8}\cdot2^{2k}\\
  -6999385995\cdot2^{14}\cdot2^{k}+4767802914\cdot2^{20} \quad \text{for} \quad k\geqslant 9 \\
  \displaystyle   \Gamma_{9}^{\left[2\atop {2\atop {2\atop{2\atop{2\atop2}}}}\right]\times k}  = 1968624\cdot2^{4k}+15196608\cdot1395\cdot2^{3k}
  -2387571795\cdot2^{12}\cdot2^{2k}\\
  +4814516070\cdot2^{18}\cdot2^{k}-2760151464\cdot2^{24} \quad \text{for} \quad k\geqslant 10 \\
    \displaystyle   \Gamma_{10}^{\left[2\atop {2\atop {2\atop{2\atop{2\atop2}}}}\right]\times k}  = 2016\cdot\big[2^{5k}+81685\cdot2^{4k}-79052480\cdot2^{3k}+2^{13}\cdot2888735\cdot2^{2k}\\-1239163\cdot2^{21}\cdot2^{k}+2^{30}\cdot82645\big] \quad \text{for} \quad k\geqslant 11 \\
  \displaystyle   \Gamma_{11}^{\left[2\atop {2\atop {2\atop{2\atop{2\atop2}}}}\right]\times k}  = 256032\cdot\big[2^{5k}-1984\cdot2^{4k}\\+1269760\cdot2^{3k}-325058560\cdot2^{2k}+31744\cdot2^{20}\cdot2^{k}-2^{40}\big] \quad \text{for} \quad k\geqslant 12 \\
 \displaystyle   \Gamma_{12}^{\left[2\atop {2\atop {2\atop{2\atop{2\atop2}}}}\right]\times k}  = 2^{6}\cdot\big[2^{6k}-63\cdot2^{6}\cdot2^{5k}+651\cdot2^{13}\cdot2^{4k}-1395\cdot2^{21}\cdot2^{3k}\\+651\cdot2^{30}\cdot2^{2k}-63\cdot2^{40}\cdot2^{k}+2^{51}\big] \quad \text{for} \quad k\geqslant 12 \\
\end{cases}
    \end{equation}
\end{thm}

\begin{proof}
The proof is just a generalization of the similar proof of \textbf{Theorem 2.1} in [13]\\
We proceed as follows :\\
To prove \eqref{eq 2.14} we apply \eqref{eq 2.8} with n=6.\\
To prove \eqref{eq 2.15} we combine  \eqref{eq 2.10} with n=6,  \eqref{eq 2.11},  \eqref{eq 2.12} and  \eqref{eq 2.13}. 
\end{proof}
\begin{example}
\textbf{Computation of $ \Gamma_{i}^{\left[2\atop {2\atop {2\atop{2\atop{2\atop2}}}}\right]\times 6} $  for $0\leqslant i \leqslant 6 $}
\begin{equation*}
 \Gamma_{i}^{\left[2\atop {2\atop {2\atop{2\atop{2\atop2}}}}\right]\times 6} = 
\begin{cases} 
1  \quad \text{for} \quad i=0 \\
189  \quad \text{for} \quad i=1 \\
35154  \quad \text{for} \quad i=2 \\
5155920  \quad \text{for} \quad i=3 \\
645271200  \quad \text{for} \quad i=4 \\
256536315\cdot 2^8  \quad \text{for} \quad i=5 \\
2^{14}\cdot 264387375  \quad \text{for} \quad i=6 \\
\end{cases}
\end{equation*}
\begin{proof}
Apply \eqref{eq 2.14} with k=6.
\end{proof}
Computing $ \Gamma_{i}^{\left[2\atop{\vdots \atop 2}\right]\times 6}$  for n=6 we obtain the same result (see [14] ).
\end{example}

\end{document}